\documentclass[article]{amsart}
\usepackage[utf8]{inputenc}

\usepackage{xcolor,fullpage,hyperref,cleveref,enumitem}
\usepackage{tikz,amsmath,amssymb,color,mathtools}
\usetikzlibrary{calc}
\usepackage{float}
\usepackage{pgf}
\usepackage{amsrefs}
\usepackage{amsthm}
\usepackage{listings}

\usepackage{framed}
\usepackage{multicol}
\usepackage{graphicx}
\usepackage{wrapfig}
\usepackage{cancel}
\usepackage{stmaryrd}
\usepackage{dsfont}
\theoremstyle{definition} 
\newtheorem{theorem}{Theorem}[section]

\newtheorem{proposition}[theorem]{Proposition}

\newtheorem{lemma}[theorem]{Lemma}
\newtheorem{conjecture}[theorem]{Conjecture}

\theoremstyle{definition}

\theoremstyle{remark}

\newcommand{\oeis}[1]{\href{https://oeis.org/#1}{\rm  \underline{#1}}}

\usepackage{thmtools} 
\usepackage{thm-restate}

\title{Experimenting with Permutation Wordle}

\author[Hiveley]{Aurora Hiveley}
\address[A.~Hiveley]{Department of Mathematics, Rutgers University, Piscataway, NJ 08854}
\email{\textcolor{blue}{\href{mailto:aurora.hiveley@rutgers.edu}{aurora.hiveley@rutgers.edu}}}

\begin{document}
\begin{abstract}
Consider a game of permutation wordle in which a player attempts to guess a secret permutation in $S_n$ in as few guesses as possible. In each round, the guessing player is told which indices of their guessed permutation are correct. How can we optimize the player's strategy? Samuel Kutin and Lawren Smithline propose a strategy called \textit{cyclic shift} in which all incorrect entries are shifted one index to the right in successive guesses, and they conjecture its optimality. We investigate this conjecture by formalizing what a ``strategy" looks like, performing experimental analysis on inductively constructed strategies, and examining the coefficients of an inductive strategy's generating function. 
\end{abstract}

\maketitle

\section{Introduction} 
Consider a game in which a game master orders $n$ objects labelled $1$ through $n$, but keeps this ordering a secret. On your turn, you will attempt to guess the secret ordering of these objects by guessing some permutation of $1, \dots, n$. The game master will then indicate the subset of these $n$ objects which are in the correct positions. The game ends when all objects are in the correct positions, meaning that you have guessed the secret ordering (referred to as the secret permutation.) This game, called ``permutation wordle," was developed by Samuel Kutin and Lawren Smithline \cite{kutin2024permutationwordle} as a modification of Josh Wardle's popular New York Times game, Wordle. It is so called due to the similarities between the feedback loop in our game and that of Wordle. 

Consider the following example game. The game master thinks of a permutation on $[5]$. Say that the guesser's first guess is 12345. The game master informs the guesser that the number 2 is in the correct position, so the latter's next guess is 52134. Now, the game master says that the number 3 is also in the correct position. The guesser's next guess is 42135, and the game master tells them that this guess is fully correct. This game is completed in three guesses. 

The guessing player in this game will be interested in an optimal strategy. Kutin and Smithline propose a strategy in which all incorrect objects from a guess are shifted one position to the right in a successive guess, skipping over (and leaving in place) each correct object. They dub this strategy \textit{circular shift}, and they conjecture that this strategy is optimal among all strategies in such a game. Kutin and Smithline say that a strategy is optimal if the probability of success in $r$ guesses, denoted by $P_r$, is maximal for any $r$. We say that a strategy is optimal if the average number of guesses needed to complete the game is minimized.

In this paper, we investigate the Kutin-Smithline's conjecture formalizing what a strategy looks like for a game of permutation wordle. We pursue an inductive approach, building strategies iteratively and examining how performance changes across iterations. The central result of this paper uses generating functions to study the optimality of cyclic shift for a game terminating in one, two, or three guesses.

\section{Notation} \label{sec:not}
Following Kutin-Smithline's notation, we will use $\gamma_1, \gamma_2, \dots $ to refer to the guessed permutations in each round. Without loss of generality and for ease of analysis, in this paper we will always assume that $\gamma_1 = [1,2,\dots,n]$. For a guess $\gamma_r$, let the set of correct positions (the ones which agree with the secret permutation) be denoted by $\mathcal{J}_r$, and let the set of incorrect positions be denoted by $\mathcal{I}_r$. In this sense, $\mathcal{J}_r \sqcup \mathcal{I}_r = [n]$ for each $r$.

We formalize a strategy $S$ by defining $S := [s_1, s_2, \dots, s_n]$ where $s_i$ is a permutation of length $i$. Throughout this paper, we refer to $s_i$ with indexing notation, i.e. $S[i]$. To form the next guess $\gamma_{r+1}$ from $\gamma_r$ using $S$, we will lock in all correct positions from $\mathcal{J}_r$ and permute the incorrect positions in $\mathcal{I}_r$. We permute the elements of $\mathcal{I}_r$ according to the strategy component $S[|\mathcal{I}_r|]$. 

In this way, we formalize Kutin-Smithline's proposed ``cyclic shift" strategy (abbreviated $CS$ in this paper) by writing the following:

\[ CS := \left[ [1], [2,1], [2,3,1], \dots, [2,3,\dots, n, 1] \right] \]

\vspace{0.5cm}

Then a general component of cyclic shift takes the form $CS[n] = [2,3,\dots, n,1]$. We will default to a cyclic shift to the \textit{right}, meaning that each incorrect index is shifted to the first available index to the right. But note that the leftward cyclic shifting strategy has an analogous interpretation in which a general component takes the form $[n,1,2,\dots,n-1]$. 

We preliminarily observe that a strategy component $S[n]$ ought to be a derangement of $[n]$, otherwise at least one incorrect entry will be guessed in the same incorrect position for $\gamma_{r+1}$ as it was in $\gamma_r$. Such a guess could never end the game since we know that at least one position is incorrect (in actuality, at least \textit{two} positions are incorrect) and we waste an opportunity to learn more information by swapping a different element into that same incorrect position. Thus we assume that such a strategy \textit{must} be inefficient, and so we use only derangements for strategy components. 

In the next section, we will investigate an inductive approach to strategy construction using derangements as strategy components. However, derangements as a whole may produce challenges for us later on. Take, for example, the derangement $[2,1,4,3]$. Say that $\gamma_r$ is generated using $S[4] = [2,1,4,3]$, then $\gamma_r$ swaps the first two entries in $\gamma_{r-1}$ as well as the last two entries. But if $\mathcal{J}_r = \emptyset$, then $S[4]$ generates $\gamma_{r+1}$ as well by performing the same pairwise swaps once more. But this means that $\mathcal{I}_{r+1} = \mathcal{I}_{r-1}$, and the game enters an infinite loop where for any $s \geq r$, the set $\mathcal{I}_s$ can never be nonempty. Hence, for simplicity of computation, later on in Section \ref{sec:gf} we will narrow our scope to include only \textit{cyclic permutations} as strategy components, thereby ensuring that no such infinite loop is possible.

\section{Strategy Analysis} \label{sec:stratanal}

Our analysis begins by compiling experimental data using the Maple package linked in the Appendix. We first construct the set of all strategies of length $n$ whose components are derangements (called ``deranged strategies") and separately construct those whose components are cyclic permutations (called ``cyclic strategies".) For each such strategy, we compute the average number of guesses needed to win a game of permutation wordle if the guesser uses that strategy. To do this, for each permutation of length $n$, we execute a game of permutation wordle, count how many guesses are needed for the game to terminate, and average this count over all $n!$ secret permutations of length $n$. 

Using Maple, we were able to verify Kutin-Smithline's conjecture for cyclic strategies up to length $n=7$ and for deranged strategies up to length $n=5$. For larger $n$, we encountered limitations due to run time and computational memory. We also tested cyclic shift's performance by comparing it to other strategies of the same length constructed by induction (see the next section for details.) For inductively constructed strategies, we were able to verify Kutin-Smithline's conjecture up to $n=8$. 

With supporting experimental findings in hand, we turn our attention to a mathematical study of cyclic shift compared to other strategies. In the following section, we will define an inductively constructed strategy and examine whether induction can help us verify cyclic shift's optimality.

\subsection{Inductive Construction} \label{sub: inductive}

We pursue an inductive approach. Consider a strategy $S$ of length $n$ such that $S[i] = CS[i]$ for $1 \leq i \leq n-1$, but $S[n]$ is an arbitrary cyclic permutation of length $n$. Then to study Kutin-Smithline's conjecture, we must focus on how $S[n]$ performs before induction takes over. In other words, we are most interested in how a game plays out up until $\mathcal{J}_r \neq \emptyset$. From Kutin-Smithline's conjecture, we would expect that $[2,3,\dots,n,1]$ (or $[n,1,2,\dots,n-1]$, for a left-shifting strategy) perform the ``best" in this context. But how can we test this theory?

We quantify the performance of a strategy component $S[n]$ on permutations of length $n$ as follows. For any derangement $d \in D_n$ acting as the secret permutation, we execute one iteration of our guessing procedure. As usual, our first guess will be the trivial permutation $[1,2,\dots,n]$. From there, generate another guess using $S[n]$ and count the number of correct positions, i.e. $|\mathcal{J}_2|$. Averaging this quantity across all derangements $d$, we produce a measurement of the performance of a strategy component $S[n]$. Note that we use derangements for this measurement so that we guarantee that $\mathcal{J}_1 = \emptyset$ after guessing the trivial permutation, thus assuring that $S[n]$ is triggered to produce a new permutation on $[n]$. If $\mathcal{J}_1 \neq \emptyset$, then induction kicks in immediately for a permutation of length less than $n$. 

Surprisingly, all possible (deranged) choices for $S[n]$ have the same average performance across secret permutations in $D_n$ under this framework. Even more interestingly, the average number of correct indices after one guess for a strategy component of length $n$ is $\frac{n-1}{n}$. We present the following proposition:

\begin{proposition} \label{prop:derange}
    For any deranged strategy component $S[n]$, the average size of $\mathcal{J}_2$ across all games of permutation wordle whose secret permutation is a derangement of length $n$ is $\frac{n}{n-1}$.
\end{proposition}

\begin{proof}
To study a game using $S[n]$, first note that if $\gamma_1 = [1,2,\dots,n]$, then $\gamma_2 = (S[n])^{-1}$, assuming that the secret permutation is a derangement. Then $\mathcal{J}_2 = \{i \in [n] \mid d[i] = (S[n])^{-1}[i] \} $, so the set of correct positions is precisely the set of positions where the derangement and the inverse of our strategy component agree. Since the inverse of a derangement is also a derangement, and inverses are in one-to-one correspondence, we proceed using a derangement $\delta = (S[n])^{-1}$. Note then that $|\mathcal{J}_2| = \left| \{ i \in [n] \mid \delta(i) = d(i) \} \right|$ for a secret permutation $d$, so the average size of $|\mathcal{J}_2|$ across all $d \in D_n$ is $\frac{1}{|D_n|} \sum_{d \in D_n} \left| \{ i \in [n] \mid \delta(i) = d(i) \} \right|$. To prove the claim that this average is equal to $\frac{n}{n-1}$, we will show the following equivalent statement:

\begin{equation} \label{eq:derange}
    \sum_{d \in D_n} \left| \{ i \mid d(i) = \delta (i) \} \right| = \frac{n}{n-1} \cdot | D_n|
\end{equation}

\vspace{0.25cm}

On the left hand side of Eq. \ref{eq:derange}, we may equivalently fix a position in our fixed derangement $\delta$ and count the number of other derangements which share that position, then sum across all possible positions. Then we have that 
$ \sum_{d \in D_n} \left| \{ i \mid d(i) = \delta(i) \} \right| 
= \sum_{i=1}^n \left| \{ d \in D_n \mid d(i) = \delta(i) \} \right| $. From there, we may calculate $| \{ d \in D_n \mid d(i) = \delta(i) \} |$ by utilizing the recurrence for $|D_n|$. Recall that $|D_n| = (n-1)(|D_{n-1}| + |D_{n-2}|)$, wherein we construct a derangement of length $n$ by choosing a position for the element $n$ and then deranging the remaining elements. There are $n-1$ possible locations for $n$. Say that $d[n] = j$. Then if $d[j] = n$ as well, the number of ways to derange the remaining elements is $|D_{n-2}|$. If $d[j] \neq n$, then the number of ways to derange the elements of $[n] \setminus \{n\}$ is $|D_{n-1}|$ since while $j$'s fixed position is occupied, we prohibit $d[j] = n$.

In our case for each fixed entry $i$, we know its location in $\delta$ so we do not need to \textit{choose} one of $n-1$ possible locations for the element $i$. Then it suffices to calculate the number of ways to derange the remaining elements, which is $|D_{n-1}| + |D_{n-2}|$. Thus, we have the following:

\begin{align*}
    \sum_{d \in D_n} \left| \{ i \mid d(i) = \delta(i) \} \right|
    &= \sum_{i=1}^n \left| \{ d \in D_n \mid d(i) = \delta(i) \} \right| \\
    &= \sum_{i=1}^n (|D_{n-1}| + |D_{n-2}|) \\ 
    &= n \cdot (|D_{n-1}| + |D_{n-2}|) \\ 
    &= n \cdot \left( \frac{|D_n|}{n-1} \right) 
\end{align*}
    
\end{proof}

We also note that the left hand side sum in Eq. \ref{eq:derange} produces the sequence 0, 2, 3, 12, 55, 318, 2163, 16952 for $1 \leq n \leq 8$ (otherwise known as \oeis{A284843}.) But interestingly, the distribution of correct indices across all derangements is not equivalent for all strategy components $S[n]$. Consider the examples in Table \ref{tab:example}. For two possible strategy components $S[n]$ where $n=4$, the corresponding set $\mathcal{J}_2$ is given for each secret permutation $d \in D_4$. For convenience, $\gamma_2 = (S[n])^{-1}$ is also listed beneath each example $S[n]$. Observe that for $CS$, there is one derangement with 4 common positions, two with 2, and four with 1. Meanwhile for the cycle $[2,1,4,3]$, there is one derangement with 4 common positions and four with 2, but none with only 1. 

In this section, we proved that all $S[n] \in D_n$ have the same average performance. Then in an inductive approach, simply trading $S[n]$ for another derangement and examining its highest level performance will not do. We must instead examine how the performance later on in the game is effected by this exchange of strategy components. In the next section, we will proceed by studying the number of permutations which can be guessed in $r$ guesses for a given number $r$.

\begin{table}[h!]
\centering

\caption{$\mathcal{J}_2$ for two cyclic $S[n]$}
\begin{tabular}{c||c|c}
\text{Secret Permutation} & $S[n] := [2,3,4,1]$ & $S[n] := [2,1,4,3]$ \\
& $\gamma_2 = [4,1,2,3]$ & $\gamma_2 = [2,1,4,3]$ \\
\hline
%%%%%%%%%%%%%% 4123 %%%%%% 2143 is own inverse
$[2,1,4,3]$ & $\{2,4\}$ & $\{1,2,3,4\}$ \\ 
$[2,3,4,1]$ & $\emptyset$ & $\{1,3\}$ \\ 
$[2,4,1,3]$ & $\{4\}$ & $\{1,4\}$ \\ 
$[3,1,4,2]$ & $\{2\}$ & $\{2,3\}$ \\ 
$[3,4,1,2]$ & $\emptyset$ & $\emptyset$ \\ 
$[3,4,2,1]$ & $\{3\}$ & $\emptyset$ \\ 
$[4,1,2,3]$ & $\{1,2,3,4\}$ & $\{2,4\}$ \\ 
$[4,3,1,2]$ & $\{1\}$ & $\emptyset$ \\ 
$[4,3,2,1]$ & $\{1,3\}$ & $\emptyset$ \\ 
\end{tabular}
\label{tab:example}
\end{table}

\section{Generating Functions} \label{sec:gf}

Kutin-Smithline discuss generating functions in their paper in the context of multi-colored permutation wordle. We do not study this extension in our paper, however we will borrow their generating function framework to analyze strategies of standard permutation wordle. For a specific strategy $S$ of length $n$, we construct a generating function $f_S(x)$ according to its performance. A term $a_r x^r$ of $f_S(x)$ communicates that there are $a_r$ many permutations of length $n$ that the strategy $S$ guesses in exactly $r$ guesses. In the rest of this section, we are particularly concerned with $a_r = [x^r] f_S(x)$ for $r = 1,2,$ and $3$.

\subsection{Linear and Quadratic Coefficients.} \label{sub:linquad}

Importantly, in section 3 of \cite{kutin2024permutationwordle}, Kutin and Smithline prove that the coefficients of the generating function for cyclic shift are the Eulerian numbers. In other words, $[x^k]f_{CS}(x) = A(n,k-1)$ if $CS$ is the cyclic shifting strategy of length $n$. Since cyclic shift's performance is determined by the number of excedances in the secret permutation, the coefficients of $f_{CS}(x)$ are nice and have this clean recurrence. Determining the coefficients of $f_S(x)$ for a more general strategy $S$ of length $n$ is more daunting, but we do have the following:

\begin{theorem} \label{thm:linquad}
    For all strategies of length $n$, $[x^1]f_{S}(x) = 1$ and $[x^2]f_{S}(x) = A(n,1)$. 
\end{theorem}

\begin{proof}
Recall that without loss of generality, $\gamma_1$ is always the trivial permutation of length $n$. Then the only permutation which can be guessed in only one guess \textit{is} the trivial permutation, and this property is independent of the strategy used. Thus, $[x^1]f_S(x)$ will always be equal to 1.

To count the number of strategies guessable in two guesses by a set strategy $S$, we will count the number of permutations that can be generated in two guesses by the strategy $S$. We start with the trivial permutation for $\gamma_1$. Some subset of indices $\mathcal{I}_1$ are incorrect after the first guess, then the strategy component $S_{|\mathcal{I}_1|}$ is used to generate the second guess. Since $S_{|\mathcal{I}_1|}$ is deterministic, counting the number of permutations which can be generated in two guesses is equivalent to counting the number of possible subsets which may form $\mathcal{I}_1$. 

Note that if $\gamma_1$ is not the secret permutation, then $2 \leq |\mathcal{I}_1| \leq n$. We cannot have zero incorrect entries, otherwise we fall into the previous case, and it is impossible to have only one incorrect entry in a permutation. Let $|\mathcal{I}_1| = k$, then there are $\binom{n}{k}$ possible subsets which can form $\mathcal{I}_1$. Allowing $k$ to range from $2$ to $n$, we have that the number of permutations which may be generated in two guesses is $\sum_{k=2}^n \binom{n}{k} = A(n,1)$ (\oeis{A000295})
    
\end{proof}

Let's consider a special case of this argument: analysis of the inductively constructed strategies from Section \ref{sub: inductive}. For an inductive strategy, any permutation with at least one correct entry in $\gamma_1$ will use a ``lower level" strategy component to generate $\gamma_2$, meaning that $\gamma_2$ is generated by the component $S[m]$ for $m < n$. However, these components are all cyclic shifting for an inductive strategy, and hence they are all identical for any inductive strategy. Then any permutation which can be guessed in two guesses by some inductive strategy, and which is \textit{not} a derangement, can be guessed by \textit{any} inductive strategy in two guesses. For a fixed inductive strategy, there is one other permutation which can be guessed in two guesses, and this is the permutation $\gamma_2 = (S[n])^{-1}$ produced when $\mathcal{J}_1 = \emptyset$. Then for any strategy, $(S[n])^{-1}$ is guessable in two turns, but any other permutation guessable in two must've had at least one common entry with $\gamma_1$, and so it is guessable by \textit{any} inductively constructed strategy. This, however, is where the similarities between cyclic shift and an arbitrary inductive strategy stop.

\subsection{Cubic Coefficients.} \label{sub:cubic}

As we discussed before in the start of Section \ref{sec:stratanal}, experimental data does suggest that certain strategies are more efficient or optimal than others, and this difference is reflected in the later coefficients of the generating function. To analyze these coefficients, we must first do some bookkeeping. In order for a game to end in at least three guesses, we must work with permutations of length 3 or greater, so our strategy must also have length at least 3. There are only two cyclic strategies of length 3: rightward cyclic shift and leftward cyclic shift. In both cases, $[x^3]f_S(x) = A(3,2) = 1$, so this case is well understood. Henceforth, we concern ourselves only with strategies of length 4 or more. Consider the following theorem:

\begin{conjecture} \label{conj:bestcubic}
    Let $f_{CS}(x)$ be the generating function for the cyclic shift strategy of length $n$. For all strategy lengths $n$, $[x^3]f_{CS}(x) > [x^3]f_{S}(x)$ for all deranged strategies $S$.
\end{conjecture} 

Experimental evidence suggests that this conjecture is true for all \textit{deranged} strategies of length up to 6. We will prove a slightly weaker version of the conjecture above which considers only cyclic strategies, but it is noteworthy that this result may be extended to deranged strategies, as well. We will prove the claim for cyclic strategies using induction, or by considering only the inductively constructed strategies described in Section \ref{sub: inductive}. Consider the following modification of the previous statement:

\begin{theorem} \label{thm:indbestcubic}
    Let $f_{CS}(x)$ be the generating function for the cyclic shift strategy of length $n$. For all strategy lengths $n$, $[x^3](f_{CS}(x)) > [x^3](f_{S}(x))$ for all \textbf{inductive} strategies $S$.
\end{theorem}

But before digging into the details of proving this theorem, we present some examples for the reader. In the table below, we compile the generating functions of several inductive strategies of lengths 4 and 5, each identified by their final component $S[n]$.

\begin{table}[h!]
    \caption{Inductive strategy generating function examples}
    \centering
    \begin{tabular}{ll|l}
         & $S[n]$ & $f_S(x)$ \\
        \hline 
        \hline 
        $n=4$ & $[2,3,4,1]$ & $x^4 + 11x^3 + 11x^2 + x$ \\
        & $[2,4,1,3]$ & $3x^4 + 9x^3 + 11x^2 + x$ \\
        & $[2,4,1,3]$ & $3x^4 + 9x^3 + 11x^2 + x$ \\
        & $[4,1,2,3]$ & $5x^4 + 7x^3 + 11x^2 + x$ \\
        \hline 
        $n=5$ & $[2,3,4,5,1]$ & $x^5 + 26x^4 + 66x^3 + 26x^2 + x$  \\
        & $[4,3,1,5,2]$ & $8x^5 + 25x^4 + 60x^3 + 26x^2 + x$ \\
        & $[3,5,2,1,4]$ & $x^6 + 10x^5 + 27x^4 + 55x^3 + 26x^2 + x$ \\
        & $[5,1,2,3,4]$ & $5x^6 + 11x^5 + 26x^4 + 51x^3 + 26x^2 + x$  \\
    \end{tabular}
    \label{tab:gf_examples}
\end{table}

Observe, as proven in Section \ref{sub:linquad}, that $[x^1]f_S(x) = 1$ and $[x^2]f_S(x) = A(n,1)$ for each strategy. The first strategy listed in each table segment is $CS$, and (of course) among these examples the cubic coefficient is maximum for $CS$ of length 4 and 5.

To count the permutations which are guessed by a strategy $S$ of length $n$ in exactly three guesses as we did before for one and two guessees, we break our game into cases according to which of the three guesses \textit{first} resulted in a correct entry. After this point, induction takes over and we will use the standard right shifting strategy to finish the game. In other words, for a game using strategy $S$ where the secret permutation is $\pi$, we want to identify $\rho_S(\pi) := \inf \{ i \mid \mathcal{J}_i \neq \emptyset \}$. Of course then $[x^3]f_S(x) = \sum_{i=1}^3 \left| \left\{ \pi \mid \rho_S(\pi) = i \right\} \right|$. We consider each case individually.

\begin{theorem} \label{thm:rho1}
    For any inductive strategy $S$ of length $n \geq 4$, $\left| \left\{ \pi \mid \rho_S(\pi) = 1 \right\} \right| = 1 - 2^{n + 1} + 3^n + \frac{n^2}{2} + \frac{5n}{2} - n\cdot 2^n$
\end{theorem}

\begin{proof}
For $\rho_S(\pi) = 1$, we have that $\mathcal{J}_1 \neq \emptyset$, meaning that the first guess (the trivial permutation) has at least one common entry with the secret permutation $\pi$. Note that $1 \leq |\mathcal{J}_1| \leq n-3$. If $|\mathcal{J}_1| = n$, then the game ends in one guess, and it is not possible for $|\mathcal{J}_1| = n-1$ since there cannot be only one entry in an incorrect location. If $|\mathcal{J}_1| = n-2$, then $\gamma_2$ swaps the two incorrect entries from $\gamma_1$, which \textit{must} produce the correct permutation, hence the game ends in two guesses. Then for the game to end in exactly three guesses, we must have that $|\mathcal{J}_1| \leq n-3$.

Say that $|\mathcal{J}_1| = k$. Then $S[n-k]$, which is of course cyclic shifting, must guess the remaining $n-k$ positions in two more guesses. The sub-permutation of $\pi$ consisting of $n-k$ remaining elements is a derangement of those elements since none of those entries were correct in $\gamma_1$. Then it suffices to count the number of derangements of length $n-k$ which can be guessed by \textit{CS} in 3 guesses total. We know from \cite{kutin2024permutationwordle} that this is equivalent to the number of derangements of length $n-k$ with two excedances. Using Lemma \ref{lem:der2ex}, we conclude that there are $2^{n-k} - (2(n-k) + 1)$ such derangements.

So, we have that there are $\binom{n}{k}$ ways to choose $\mathcal{J}_1$, and $2^{n-k} - (2(n-k) + 1)$ permutations of the remaining elements which can be guessed in the remaining two guesses. Summing over all possible values of $k$, we have that there are $\sum_{k=1}^{n-3} \binom{n}{k} \left(2^{n-k} - (2n - 2k + 1) \right)$ permutations with $\rho_S(\pi) = 1$. With the aid of Maple, we obtain the closed form $1 - 2^{n + 1} + 3^n + \frac{n^2}{2} + \frac{5n}{2} - n\cdot 2^n$ for this sum of binomials, and this sequence has been added to the OEIS as entry \oeis{A385588}.
\end{proof}

Before we proceed, we must include the derangement-counting lemma. Derangements which are guessed by $CS$ in three guesses have two excedances, according to \cite{kutin2024permutationwordle}. Such derangements are counted by \oeis{A046739}, but we include the following argument since the approach used in the proof will also be utilized later on in this paper.

\begin{lemma} \label{lem:der2ex}
    There are $2^n - (2n+1)$ derangements of length $n$ which can be guessed by $CS$ in three guesses.
\end{lemma}

\begin{proof}
Consider a derangement of $[n]$ which can be guessed by $CS$ in exactly three guesses. For a secret permutation which is a derangement, we know that $\mathcal{J}_1$ is empty and that $\gamma_2 = n 12  \dots n-1$ after right shifting. Of course $\gamma_3$ is deterministic since our game ends after three guesses, so it suffices to count the number of legal possibilities for $\mathcal{J}_2$. For a permutation of length $n$, there are $2^n$ possible sets $\mathcal{J}_2$. We know that we cannot have $|\mathcal{J}_2| = n$ or the game ends with only two guesses, and $|\mathcal{J}_2| = n-1$ is impossible, so we subtract these possibilities to obtain $2^n - \binom{n}{n} - \binom{n}{n-1} = 2^n - (n + 1)$. 

Additionally, we cannot have that $\mathcal{I}_2 = \{i,(i\mod n)+1\}$. For simplicity of notation, we will refer to $(i \mod n) + 1$ as $i+1$, with the understanding that $i=n$ results in $i+1 = 1$. If $\mathcal{I}_2 = \{i,i+1\}$, we note that $i,i+1 \notin \mathcal{J}_1$, so $i$ and $i+1$ are both shifted right in $\gamma_2$. This means that $\gamma_2 [i+1] = i$. And since $\mathcal{I}_2 = \{i,i+1\}$, then $\gamma_3$ is generated by swapping $\gamma_2[i]$ and $\gamma_2[i+1]$. This places $i = \gamma_3[i]$, but this cannot possibly be the correct permutation since $i \notin \mathcal{J}_1$. Hence this scenario produces a contradiction. Observe that this contradiction is \textit{only} produced when $\mathcal{I}_2$ consists of two adjacent entries, otherwise entry $i$ in position $i+1$ would be swapped into some other position which is not $i$, which is legal.

There are $\binom{n}{1}$ ways to choose $i$ in this scenario, so there are an additional $n$ possible sets which cannot form $\mathcal{J}_2$. This leaves us with $2^n - (2n+1)$ possible derangements with two excedances.

\end{proof}

Observe that Theorem \ref{thm:rho1} and the subsequent proof are entirely independent of the strategy component $S[n]$ since the $n$-th strategy component is never utilized. This case relies purely on induction, meaning that new guesses are generated only by using the lower index, cyclic shifting components of the strategy. We will temporarily skip the case $\rho_S(\pi) = 2$ and first consider the case of $\rho_S(\pi) = 3$ as it behaves similarly to this previous case. Note that the following claim is actually stronger than we need it to be since it covers all \textit{cyclic} strategies, not just the subset of inductive ones. 

\begin{theorem} \label{thm:rho3}
    For any cyclic strategy $S$ of length $n \geq 4$, $\left| \left\{ \pi \mid \rho_S(\pi) = 3 \right\} \right| = 1 $
\end{theorem}

\begin{proof}
For $\rho_S(\pi) = 3$, we have that $\mathcal{J}_1, \mathcal{J}_2 = \emptyset$, but $\mathcal{J}_3 \neq \emptyset$. For the game to end at guess three, we need $\mathcal{J}_3 = [n]$. If $\mathcal{J}_1, \mathcal{J}_2 = \emptyset$, then $\gamma_3$ consists of the trivial permutation of length $n$ with $S[n]$ applied twice, and since $\mathcal{J}_3 = [n]$, we have that $\pi = \gamma_3$. This is the only possible outcome. There also cannot be an $i \in [n]$ such that $\gamma_1[i]$ or $\gamma_2[i] = \gamma_3[i]$ since the component $S[n]$ must be a cyclic permutation of length at least 3, and thus it cannot have any 1- or 2-cycles in it which would allow an element $i$ to return to the same location after $S[n]$ is applied twice. Thus this is a legal game with only one possible permutation.
\end{proof}

In the previous two cases, the number of permutations was independent of the strategy component $S[n]$. But we know experimentally (in Table \ref{tab:gf_examples}, for instance) that $[x^3]f_S(x)$ is not a constant for all of our inductive strategies, and since $[x^3]f_S(x) = \sum_{i=1}^3 \left| \left\{ \pi \mid \rho_S(\pi) =i \right\} \right|$, we conclude that the adjudicator is the so far unexamined set $\{ \pi \mid \rho_S(\pi) = 2\}$. First, consider the following:

\begin{theorem} \label{thm:cs_rho2}
    For the cyclic shifting strategy $CS$ of length $n \geq 4$, $\left| \left\{ \pi \mid \rho_{CS}(\pi) = 2 \right\} \right| = 2^n - 2n - 2$
\end{theorem}

\begin{proof}
In this case where $\rho_S(\pi) = 2$, $\mathcal{J}_1 = \emptyset$, so $\gamma_2 = [n,1,2,\dots,n-1]$ after a single right shift. As before, note that $|\mathcal{J}_2| \neq n, n-1$, and that $\gamma_3$ is deterministic since the game ends in three guesses, so it suffices to count the number of legal possible sets $\mathcal{J}_2$. Once again, notice that there cannot be two adjacent entries $i,i+1 \in \mathcal{I}_2$ (simplifying our notation as in Lemma \ref{lem:der2ex}.) If such an $i$ existed, then $\gamma_2[i+1] = i$, and by the same argument as before the entry $i$ is right shifted back around so that $\gamma_3[i] = i$. But again, this produces a contradiction since we assumed that the game ended in three guesses. %, but $i \notin \mathcal{J}_1$ since $\mathcal{J}_1 = \emptyset$.

Any other subset of $[n]$ represents a legal possibility for $\mathcal{J}_2$, so we count them as follows. There are $2^n$ total possible subsets, but we cannot have subsets of size 0 or of size $n$ or $n-1$. Then we have $2^n - \binom{n}{0} - \binom{n}{n} - \binom{n}{n-1} = 2^n - 1 - 1 - n = 2^n - n - 2$. We must also subtract off subsets of size two which consist of adjacent pairs $\{i,i+1\}$. There are $\binom{n}{1} = n$ choices for $i$, so we subtract these off as well and obtain the desired result: $2^n - n - 2 - n = 2^n - 2n - 2$.
    
\end{proof}

To demonstrate the optimality of $CS$ compared to any other inductive strategy $S$ in games lasting for only three guesses, we aimed to show that $[x^3]f_S(x) < [x^3]f_{CS}(x)$. But with the previous two calculations in hand, we actually have that $\left| \left\{ \pi \mid \rho_{S}(\pi) = i \right\} \right|$ are fixed when $i \in \{ 1,3\}$ for any inductive strategy of length $n$. It then suffices to show that $\left| \left\{ \pi \mid \rho_{S}(\pi) = 2 \right\} \right| < 2^n - 2n - 2$ for any arbitrary $S$. 

\begin{theorem} \label{thm:bestrho2}
    For any inductively constructed strategy $S$ of length $n \geq 4$, $\left| \left\{ \pi \mid \rho_{S}(\pi) = 2 \right\} \right| < 2^n - 2n - 2$
\end{theorem}

\begin{proof}
As before, we want to count the number of possible $\mathcal{J}_2$ subsets. Note, of course, that the number of these subsets cannot exceed $2^n - n - 2$ since regardless of the strategy used, $\mathcal{J}_2$ must have a size at least 1 and no bigger than $n-2$ for $\rho_S(\pi) = 2$ in a game ending after three guesses. Then we must count the number of outlying subsets which are also subtracted. 

Recall that since $\mathcal{J}_1 = \emptyset$, $\gamma_2 = (S[n])^{-1}$, and since $S[n]$ is a cyclic permutation (and hence a derangement), the inverse is also a derangement. In the case of standard cyclic shift studied before, the game wouldn't end after three guesses if $\mathcal{J}_2 = \{i, (i \mod n)+1 \}$ since entry $i$ would cycle back around to position $i$ in $\gamma_3$. Now, let $i$ be in position $j$ of $\gamma_2$, noting that $i \neq j$ since $\gamma_2$ is a derangement. Similarly, let $k$ be the position of entry $j$ in $\gamma_2$. Then if $\mathcal{J}_2 = \{ j,k \}$, then entry $j$ is shifted back to position $j$ in $\gamma_3$ when cyclic shift takes over, which produces a contradiction as before since the game cannot end in three guesses. 

This is also true for any superset of $\{j,k\}$ which contains only indices between $j$ and $k$, as the entry $j$ will still circle back around to index $j$. For example, if $j = 2$ and $k=5$, then the sets $\{2,3,5\}$ and $\{2,4,5\}$ are also illegal since entry 2 in position 5 will right shift back to position 2. (This was not the case for standard cyclic shift since $j = i+1$ and thus for entry $i$ to cycle back to position $i$, the subset $\mathcal{J}_2$ \textit{had} to have size exactly 2.) Clearly for every choice of $j$ (of which there are $n$ total), there is at least one set subtracted, so we have at \textit{most} $2^n - 2n - 2$ legal sets $\mathcal{J}_2$. However, there must be at least one choice of $j$ such that $k \neq j+1$, and hence there are sets containing $\{j,k\}$ and elements between $j$ and $k$ which must also be subtracted. If there were no such $j$, then $\gamma_2[j+1] = j$ for all $j \in [n]$, and this would make $S[n] = [2,3,\dots, n,1]$, which is a contradiction to $S \neq CS$. Hence, we must eliminate at least one more illegal subset $\mathcal{J}_2$, and thus we have a total number of legal sets $\mathcal{J}_2$ which is \textit{strictly less} than $2^n - 2n - 2$. So for an inductive strategy $S$ where $S[n] \neq [2,3,\dots,n,1]$, $\left| \left\{ \pi \mid \rho_{S}(\pi) = 2 \right\} \right| < 2^n - 2n - 2$, as desired.

\end{proof}

Thus, any inductive strategy has fewer permutations which can be guessed in three guesses (such that the first correct entry is obtained on guess number two) than standard cyclic shift does. Since we proved that $\left| \left\{ \pi \mid \rho_{S}(\pi) = 2 \right\} \right| < \left| \left\{ \pi \mid \rho_{CS}(\pi) = 2 \right\} \right|$, we also get that $[x^3]f_S(x) < [x^3]f_{CS}(x)$, meaning that cyclic shift performs optimally for games lasting only three guesses when compared to other inductively constructed strategies.

\subsection{Worst Cubic Coefficients.} \label{sub:worstcubic}

In the same context as the previous section, the inductive strategy whose performance is \textit{least} optimal is the strategy with $S[n] = [n,1,2,\dots,n-1]$, i.e. which \textit{left} shifts for a permutation of length $n$, but then inductively \textit{right shifts} for all smaller sub-permutations. We will refer to this strategy as $CSL$. Notice that Table \ref{tab:gf_examples} includes the generating functions for $CSL$ of lengths 4 and 5 in the last rows of each subsection, respectively. In fact, we can use Maple to experimentally generate the cubic coefficients for $CSL$ of lengths 3 through 8 as follows: $1,7,51,263,1100,4093$. While the cubic coefficients for $CS$ were equal to $A(n,2)$, this sequence does not have as neat of a closed form definition. But we can still calculate $[x^3]f_{CSL}(x)$ by considering the same three cases from the previous section, and we will do so in this section.

At first glance, we \textit{do} notice that for $n \geq 4$, each of these cubic coefficients is strictly smaller than $A(n,2)$. This indicates that there are strictly more permutations which $CS$ can guess in exactly three guesses than $CSL$ can. After comparing with the cubic coefficients of \textit{all} other inductive strategies, we descry the following:

\begin{theorem} \label{thm:worstcubic}
    For all strategy lengths $n \geq 4$, $[x^3]f_{CSL}(x) < [x^3]f_{S}(x)$ for all inductive strategies $S$.
\end{theorem}

% Unlike for standard cyclic shift, this result does not immediately extend to deranged strategies
Note, as before, that this claim is equivalent to the statement $\left| \left\{ \pi \mid \rho_{CSL}(\pi) = 2 \right\} \right| < \left| \left\{ \pi \mid \rho_{S}(\pi) = 2 \right\} \right|$ for any inductive strategy $S$. Then we first count the permutations on the left hand side of our inequality as follows:

\begin{theorem} \label{thm:csl_rho2}
    For the inductive strategy $CSL$ of length $n \geq 4$, $\left| \left\{ \pi \mid \rho_{CSL}(\pi) = 2 \right\} \right| = L_n - n - 1$, where $L_n$ is the $n$-th Lucas number.
\end{theorem}

\begin{proof}
As before, it suffices to count the number of legal possible sets $\mathcal{J}_2$. Notice that there \textit{still} cannot be two adjacent entries $i,i+1 \in \mathcal{I}_2$ (letting $i+1$ stand in for $(i \mod n) + 1$, as usual.) If such an $i$ existed, then $\gamma_2[i] = i+1$ after left shifting, and since $i,i+1 \in \mathcal{I}_2$, the entry $i+1$ in position $i$ is right shifted back to position $i+1$ in $\gamma_3$. But $\gamma_3[i+1] = i+1$ produces a contradiction since we assumed that the game ended in three guesses and $i+1 \notin \mathcal{J}_1$ since $\mathcal{J}_1 = \emptyset$. 

Then we may count legal subsets $\mathcal{J}_2$ with size at least 1 and at most $n-2$ by instead counting sets $\mathcal{I}_2$ with size at least 2 and at most $n-1$ which contain no pairs of adjacent elements. These are counted by \oeis{A023548}, where we see that the number of such sets from a permutation of length $n$ is $L_n - n - 1$, where $L_n$ is the $n$-th Lucas number.

\end{proof}

With this count in hand, we now aim to show that $\left| \left\{ \pi \mid \rho_{S}(\pi) = 2 \right\} \right| > \left| \left\{ \pi \mid \rho_{CSL}(\pi) = 2 \right\} \right|$ by considering the number of legal sets $\mathcal{J}_2$. We will use a similar strategy to that employed in Theorem \ref{thm:bestrho2}.

\begin{theorem} \label{thm:worstrho2}
    For any inductive strategy $S$ of length $n \geq 4$, $\left| \left\{ \pi \mid \rho_{S}(\pi) = 2 \right\} \right| > \left| \left\{ \pi \mid \rho_{CSL}(\pi) = 2 \right\} \right|$
\end{theorem}

\begin{proof}
Consider an arbitrary inductive strategy $S$ and a game of permutation wordle with the secret permutation $\pi$ ending in three guesses but such that $\rho_S(\pi) = 2$. Then $\gamma_2 = (S[n])^{-1}$, as we have seen before. From $\gamma_2$, construct a new permutation $p$ as follows. Let $p[1] = 1$. Then let $p[2] = k_1$ where $\gamma_2[k_1] = 1$. Next, let $p[3] = k_2$ where $\gamma_2[k_2] = k_1$. Repeat up to $j = k_{n-2}$, at which point all elements of $[n]$ will appear in $p$. Since $S[n]$ is a cyclic permutation, so is $(S[n])^{-1} = \gamma_2$, and thus $p$ is a valid, cyclic permutation of length $n$. More simply, throughout this paper we have defined the permutations used in strategy components using the bottom line of two-line notation. Now, $p$ is the same permutation as $\gamma_2$, but written in cycle notation.

Now recall that in Theorem \ref{thm:bestrho2}, we observed that for any inductive strategy $S$, we cannot have $\mathcal{I}_2 = \{j,k\}$ such that $\gamma_2[j] = i$ and $\gamma_2[k] = j$. So, to count legal subsets $\mathcal{I}_2$, we must count subsets that do not contain $\{j,k\}$. This is equivalent to counting subsets of $p$ which do not contain adjacent elements since each adjacent pair is precisely a duo $\{j,k\}$ where $\gamma_2[k] = j$. Using the same argument from Theorem \ref{thm:csl_rho2} to count such subsets, we have at least $L_n - n - 1$ possible sets $\mathcal{I}_2$. Thus $|\{ \pi \mid \rho_{S}(\pi) = 2\}| \geq L_n - n - 1$. 
% \sum_{k=2}^{\floor{(n+1)/2}} \binom{n+1-k}{k} - \sum_{k=0}^{\floor{(n-3)/2}} \binom{n-3-k}{k}

But we have actually undercounted, here. Before, in Theorem \ref{thm:csl_rho2} when counting legal $\mathcal{J}_2$'s for $CSL$, \textit{any} superset of $\{i,i+1\}$ produced a contradiction and thus could not be a legal $\mathcal{I}_2$ set. In Theorem \ref{thm:bestrho2}, we noted that for a generic strategy $S$, we need only eliminate supersets of $\{j,k\}$ which also contain elements $l$ such that $j < l < k$. In other words, a set $\{j,k,m\}$ where $k < m < j$ (modulo $n$) is a totally acceptable candidate set for $\mathcal{I}_2$ since the entry $j$ will still be right shifted back to position $j$, which was previously occupied by the element $k$, when $\gamma_3$ is generated. 

In $CSL$, $k = j-1$ for all $j \in [n]$. Hence, there are no possible $m$ such that $k < m < j$. Then for any inductive strategy $S$ which is \textit{not} $CSL$, there must be at least one $j$ such that its corresponding $k \neq j-1$. Such a $j$ has at least one possible $m$, and for each possible $m$ there is at least one set $\{j,k,m\}$ which is a legal $\mathcal{I}_2$, and hence represents a permutation $\pi$ such that $\rho_S(\pi) = 2$. Then for any arbitrary inductive strategy $S$, $|\{ \pi \mid \rho_{S}(\pi) = 2\}| \geq (L_n - n - 1) + 1$, or $|\{ \pi \mid \rho_{S}(\pi) = 2\}| > L_n - n - 1$. Thus $|\{ \pi \mid \rho_{S}(\pi) = 2\}| > |\{ \pi \mid \rho_{CSL}(\pi) = 2\}|$, as desired.

\end{proof}

Since $|\{ \pi \mid \rho_{S}(\pi) = 2\}| > |\{ \pi \mid \rho_{CSL}(\pi) = 2\}|$ and we know that $|\{ \pi \mid \rho_{S}(\pi) = 1\}|$ and $|\{ \pi \mid \rho_{S}(\pi) = 3\}|$ are equal for all inductive strategies $S$, we conclude that $[x^3]f_{CSL}(x) < [x^3]f_S(x)$. This proves Theorem \ref{thm:worstcubic}. And, as promised, we may now use the fact that $[x^3]f_{CSL}(x) = \sum_{i=1}^3 |\{\pi \mid \rho_{CSL}(\pi) = i\}|$ to calculate $[x^3]f_{CSL}(x)$. Adding together the results of Theorems \ref{thm:rho1}, \ref{thm:rho3}, and \ref{thm:worstrho2}, we have:

\[ 
[x^3] f_{CSL}(x) = 
1 + \left( 1 - 2^{n + 1} + 3^n + \frac{n^2}{2} + \frac{5n}{2} - n\cdot 2^n \right) + \left( L_n - n - 1 \right)
\]

\vspace{0.25cm}

For $3 \leq n \leq 8$, we can use this equation to replicate the results produced by Maple at the beginning of Section \ref{sub:worstcubic}.

\section{Conclusion \& Future Work} \label{sec:conc}

Thus far, we have proven Kutin-Smithline's conjecture for games terminating in exactly three guesses. In Section \ref{sec:gf}, we used the coefficients of a strategy's generating function to show inductively that $CS$ is optimal for a game terminating in exactly three guesses, showing that $CS$ is able to guess the maximum number of permutations in exactly three guesses. It is possible that the techniques in this paper could be extended to games terminating in four or more guesses, however the work in this paper surrounding cubic coefficients was motivated by experimental observations which do not extend to coefficients of higher order terms in our generating functions. So, it may be that a more nuanced approach is required to analyze and compare those coefficients.

Even so, the results in this paper concern a specific case of the work originally conducted by Kutin and Smithline. We narrowed our focus to consider games terminating in one, two or three guesses, and we also made several assumptions and simplifications when defining what constitutes a ``strategy." Our strategies consisted primarily of cyclic permutations, so expanding this to derangements, or even to general permutations, will likely change our findings. We also required that a strategy component $S[n]$ was fixed within a strategy $S$, meaning that there was no allowance for a guesser to modify their strategy midway through, perhaps after several failed guesses resulting in $\mathcal{I}_r = \emptyset$. So there is still a large amount of work to be done investigating cyclic shift's performance and determining whether it truly \textit{is} optimal in a game of permutation wordle, and not just in a game with the limitations and assumptions employed in our paper.

\section*{Appendix} \label{sec:appen}
The majority of the findings in this paper are supported by a Maple package, which can be found \href{https://aurorahiveley.github.io/wordle.txt}{here}. This package includes examples as well as a handful of procedures which verify (experimentally) each of the theorems in this paper. Any bugs should be reported to the author at aurora.hiveley@rutgers.edu.

\section*{Acknowledgements}
The author thanks her advisor Dr. Doron Zeilberger for the introduction to the problem and feedback on earlier drafts.

\bibliographystyle{plain}
\bibliography{sources}

\end{document}